\documentclass[12pt,a4paper,reqno]{amsart}
\usepackage[usenames, dvipsnames]{color}
\usepackage{tikz}
\usepackage{amssymb}
\usepackage{amsthm,scalerel}
\usepackage{enumerate}
\usepackage[all]{xy}
\usepackage{fancybox}
\usepackage[colorlinks, linktocpage, citecolor = blue, linkcolor = blue, urlcolor = blue]{hyperref}

\usepackage{graphicx}
\usepackage[parfill]{parskip}
\newtheorem{theorem}{Theorem}
\newtheorem{proposition}{Proposition}[section]
\newtheorem{lemma}[proposition]{Lemma}

\theoremstyle{definition}
\newtheorem{definition}[proposition]{Definition}
\newtheorem{remark}[proposition]{Remark}

\newcommand{\PP}{\mathbb{P}}
\newcommand{\p}{\mathbb{P}}

\renewcommand{\k}{\mathrm{k}}
\newcommand{\Aut}{\mathrm{Aut}}

\newcommand{\PGL}{\mathrm{PGL}}

\newcommand{\C}{\mathbb{C}}

\newcommand{\G}{\mathbb{G}}

\newcommand{\id}{\mathrm{id}}

\renewcommand{\epsilon}{\varepsilon}
\renewcommand{\phi}{\varphi}

\DeclareMathOperator{\Bir}{Bir}

\newcommand{\Dec}{\mathrm{Dec}}
\newcommand{\Ine}{\mathrm{Ine}}

\title[The decomposition groups of conics and rational cubics]{The decomposition groups of plane\\ conics and plane rational cubics}

\author{Tom Ducat}
\address{Tom Ducat\\
Research Institute for Mathematical Sciences\\
Kyoto University\\
\mbox{Kyoto} 606-8502 Japan}
\email{taducat@kurims.kyoto-u.ac.jp}

\author{Isac Hed\'en}
\address{Isac Hed\'en\\
Research Institute for Mathematical Sciences\\
Kyoto University\\
\mbox{Kyoto} 606-8502 Japan}
\email{Isac.Heden@kurims.kyoto-u.ac.jp}

\author{Susanna Zimmermann}
\address{Susanna Zimmermann\\
Universit\'e Toulouse Paul Sabatier\\ 
Institut de Math\'ematiques\\
118 route de Narbonne\\
31062 Toulouse Cedex 9}
\email{susanna.zimmermann@math.univ-toulouse.fr}

\thanks{The first and second named authors are International Research Fellows of the Japanese Society for the Promotion of Sciences, and this work was supported by Grant-in-Aid for JSPS Fellows Number 15F15771 and 15F15751 respectively. The last named author gratefully acknowledges support by the Swiss National Science Foundation grant P2BSP2\_168743.}

\begin{document}

\begin{abstract}
The decomposition group of an irreducible plane curve $X\subset\p^2$ is the subgroup $\Dec(X)\subset\Bir(\p^2)$ of birational maps which restrict to a birational map of $X$. We show that $\Dec(X)$ is generated by its elements of degree $\leq2$ when $X$ is either a conic or rational cubic curve.
\end{abstract}


\maketitle

\section{Introduction}
\subsection{Preliminaries.}

We work over an algebraically closed field $\k$ of any characteristic. By \emph{elementary quadratic transformation} we will mean a birational map $\phi\in\Bir(\p^2)$ of degree 2 with only proper base points. 

\begin{definition}
For an irreducible curve $X\subset \p^2$, the \emph{decomposition group $\Dec(X)$ of $X$} is the subgroup of $\Bir(\p^2)$ of all birational maps $\phi\in\Bir(\p^2)$ which restrict to a birational map $\phi\mid_{X}\colon X\dashrightarrow X$. 

Similarly, the \emph{inertia group $\Ine(X)$ of $X$} is the subgroup of $\Bir(\p^2)$ of all birational maps $\phi\in\Bir(\p^2)$ which restrict to the identity map $\phi\mid_{X} = \id_X$.
\end{definition}

Elements of $\Dec(X)$ are said to \emph{preserve} the curve $X$, whilst elements of $\Ine(X)$ are said to \emph{fix} $X$. We will write $\Aut(\p^2,X)=\Dec(X)\cap \PGL_3$ for the subgroup of linear maps $\Aut(\PP^2)=\PGL_3$ which preserve $X$.

The focus of this paper is on the group $\Dec(X)$ in the case that $X\subset\p^2$ is a plane rational curve of degree $\leq3$. In this case $X$ is either a line, a smooth conic, a nodal cubic or a cuspidal cubic.

\begin{remark}
A line $X\subset \p^2$ (resp.\ conic, nodal cubic, cuspidal cubic) is projectively equivalent to any other line $X'\subset \p^2$ (resp.\ conic, nodal cubic, cuspidal cubic), i.e.\ there is an automorphism $\lambda\in\PGL_3$ with $\lambda(X)=X'$. For rational curves of degree $\geq4$ this is no longer true in general.
\end{remark}

\subsection{Motivation.}
The decomposition and inertia groups of plane curves have appeared in a number of places.

\subsubsection{Decomposition and inertia groups of plane curves of genus $\geq1$}

The inertia groups of plane curves of geometric genus $\geq2$ were studied by Castelnuovo~\cite{Cas1892}, and his results were made more precise by Blanc--Pan--Vust~\cite{BPV2008}. 
In both articles adjoint linear systems are used to study properties of the group---a technique which does not work for curves of genus $\leq1$.
The inertia groups of smooth cubic curves have been studied by Blanc~\cite{B2008}.

Decomposition groups were introduced by Gizatullin~\cite{giz}, who used them as a tool to give sufficient conditions for $\Bir(\p^2)$ to be a simple group. This group is not simple, as shown later by Cantat--Lamy \cite{CL} for algebraically closed fields, and by Lonjou \cite{lonjou} for arbitrary fields. 
The decomposition groups of plane curve of genus $\geq2$ and some plane curves of genus $1$ (smooth cubic curves and Halphen curves) are described in \cite{BPV2009}, as well as the decomposition group of rational plane curves $X\subset\p^2$ of Kodaira dimension $\kappa(\p^2,X)=0$ or $1$. 

For curves $X\subset\p^2$ with $\kappa(\p^2,X)=-\infty$, the pair $(\p^2,X)$ is birationally equivalent to $(\p^2,L)$ where $L\subset\p^2$ is a line, and a description of $\Dec(L)$ is given by Theorem~\ref{lineTheorem} below. As $X\subset\p^2$ is the image of $L$ under a birational transformation $\phi$ of $\p^2$, we have an isomorphism $\Dec(X)\simeq\Dec(L)$, given by  $\psi\mapsto\phi^{-1} \psi \phi$. Although it is not degree-preserving, this isomorphism shows that $\Dec(X)$ is not finite.

\subsubsection{The decomposition group of a line} 

The classical Noether--Castelnuovo Theorem~\cite{Cas} states that the Cremona group $\Bir(\p^2)$ has a presentation given by:
\[ \Bir(\p^2) = \big\langle \, \PGL_3, \, \sigma \, \big\rangle \]
where $\sigma$ is any choice of elementary quadratic transformation. The second two authors~\cite{hedzim} have shown that an analogous statement holds for the decomposition group of a line:
\begin{theorem}[\cite{hedzim}]\label{lineTheorem}
Let $L\subset \p^2$ be a line. Then
\[ \Dec(L) = \big\langle \, \Aut(\p^2,L), \, \sigma \, \big\rangle \]
for any choice of elementary quadratic transformation $\sigma\in\Dec(L)$. In particular any map $\tau\in\Dec(L)$ can be factored into elementary quadratic transformations inside $\Dec(L)$.
\end{theorem}

In this article, we present a similar theorem for conic and rational cubic curves. Uehara~\cite[Proposition 2.11]{uehara} proves that for the cuspidal cubic $X\subset\p^2$, the elements of the subset
\[\{f\in\Dec(X)\mid \ \text{$f$ is an automorphism near the cusp}\}\subsetneq\Dec(X)\]
 can be decomposed into quadratic transformations preserving $X$. 
Theorem~\ref{cubicTheorem} generalises his result to all of $\Dec(X)$.

\subsubsection{Relationship to dynamics of birational maps}

Birational maps of $\p^2$ preserving a curve of degree $\leq3$ show up naturally when studying the dynamical behaviour of birational maps of surfaces. For instance, Diller--Jackson--Sommese~\cite[Theorem 1.1]{DJS} show that a connected curve which is preserved by an algebraically stable element of $\Bir(\p^2)$ with positive first dynamical degree necessarily has degree $\leq3$.

In their studies of automorphisms of rational surfaces, Bedford--Kim~\cite[\S~1]{BK} explore the dynamical behaviour of the family of birational transformations $f_{a,b}\colon(x,y)\mapsto\left(y,\frac{y+a}{x+b}\right)$, for $a,b\in\C$. In particular, they focus on maps of this kind preserving a curve, and show that this curve is necessarily cubic.

\subsection{Main results}

We will use Theorem~\ref{lineTheorem} to deduce:

\begin{theorem}\label{conicTheorem}
Let $C\subset \p^2$ be a conic. Then any map $\tau\in\Dec(C)$ can be factored into elementary quadratic transformations inside $\Dec(C)$.
\end{theorem}

Moreover, from Theorem~\ref{conicTheorem} we will deduce:

\begin{theorem}\label{cubicTheorem}
Let $X\subset \p^2$ be a rational cubic and suppose that the characteristic of $\k$ is not 2. Then any map $\tau\in\Dec(X)$ can be factored into elementary quadratic transformations inside $\Dec(X)$.
\end{theorem}

The basic strategy used to prove both Theorems~\ref{conicTheorem} \& \ref{cubicTheorem} is the same in each case and is explained in \S~\ref{mainPropSection}. Given a curve $Z\subset \p^2$, the idea is to conjugate $\tau\in\Dec(Z)$ to $\tau'\in\Dec(Y)$, for a curve $Y\subset\p^2$ of lower degree, and then use the result for $Y$.

\begin{remark} \label{notChar2}
The proof of each theorem is elementary and only requires choosing quadratic transformations with base points that lie outside of a collection of finitely many points and lines. In the cubic case we need to choose base points which avoid all of the tangent lines to a conic which pass through a given point. We must restrict to a field $\k$ of characteristic $\neq2$ in this case, since over fields of characteristic 2 every line through a given point may be tangent to a conic (see \cite[Appendix to \S~2]{conicsInChar2}).
\end{remark}

\begin{remark} 
As shown in Proposition~\ref{countablyGeneratedProp}, for a conic $C$ it is still possible to write $\Dec(C) = \big\langle \Aut(\p^2,C), \sigma \big\rangle$ using just one suitably general elementary quadratic transformation $\sigma$ (where `suitably general' means that $\sigma$ does not contract a tangent line to $C$). However, if the base field $\k$ is uncountable then we need an uncountable number of elementary quadratic transformations to generate both $\Ine(C)$ (see Remark~\ref{uncountableIneC}) and $\Dec(X)$ for $X$ a nodal cubic (see \S~\ref{generatingCubic}).
\end{remark}
\subsection{Acknowledgements} We would like to thank Eric Bedford and Jeffrey Diller for helpful comments.

\section{The main Proposition} \label{mainPropSection}

Let $Y,Z\subset \p^2$ be two arbitrary irreducible plane curves. 

\begin{definition}
Let $\Phi_{Y,Z} \subset \Bir(\p^2)$ be the \emph{set of all elementary quadratic transformations $\phi$ which map $Y$ birationally onto $Z$}. 
\end{definition}
Note that $\Phi_{Y,Z}$ is a (possibly empty) \emph{subset} of  $\Bir(\p^2)$ and not a subgroup. For any $\phi,\psi\in\Phi_{Y,Z}$ we clearly have $\phi\psi^{-1}\in\Dec(Z)$. More generally for any $\tau\in \Dec(Y)$ we have $\phi\tau\psi^{-1}\in\Dec(Z)$.
 
\begin{proposition} \label{mainProp}
Suppose that $\Phi_{Y,Z}\neq\emptyset$ and the following three statements hold:
\begin{enumerate}
\item[(A)] Any $\tau\in \Dec(Y)$ can be factored into elementary quadratic transformations inside $\Dec(Y)$.
\item[(B)] For any $\phi,\psi \in\Phi_{Y,Z}$ the composition $\phi\psi^{-1}\in \Dec(Z)$ can be factored into elementary quadratic transformations inside $\Dec(Z)$.
\item[(C)] For any elementary quadratic transformation $\tau\in\Dec(Y)$ there exist $\phi,\psi \in\Phi_{Y,Z}$ such that $\phi \tau \psi^{-1}\in \Dec(Z)$ can be factored into elementary quadratic transformations inside $\Dec(Z)$.
\end{enumerate}
Then any $\tau\in \Dec(Z)$ can be factored into elementary quadratic transformations inside $\Dec(Z)$. 
\end{proposition}

\begin{proof}
Suppose that $\tau\in\Dec(Z)$ and choose any two maps $\phi,\psi\in\Phi_{Y,Z}\neq\emptyset$. Then by (A) we can factor $\tau' := \psi^{-1}\tau\phi\in\Dec(Y)$ into elementary quadratic transformations $\tau' = \tau_n\tau_{n-1}\cdots\tau_2\tau_1$ with $\tau_i\in\Dec(Y)$ for all $i=1,\ldots,n$. 

By (C) we can find $\phi_i,\psi_i\in\Phi_{Y,Z}$ such that $f_i := \phi_i\tau_i\psi_i^{-1}\in \Dec(Z)$ can be factored into elementary quadratic transformations inside $\Dec(Z)$ for all $i=1,\ldots,n$. 

Now let $\phi_0 := \phi$ and $\psi_{n+1} := \psi$. Then by (B) we can factor $g_i := \psi_{i+1}\phi_{i}^{-1}\in \Dec(Z)$ into elementary quadratic transformations inside $\Dec(Z)$ for all $i=0,\ldots,n$.  

We can write $\tau = g_n f_n g_{n-1} \cdots g_1  f_1  g_0$, according to the diagram:
\begin{center}
\begin{tikzpicture}[scale=1]
	\node at (0,0) {$Z$}; 	\node at (2,0) {$Z$}; 	\node at (4,0) {$Z$};
	\node at (6,0) {$Z$}; 	\node at (8,0) {$Z$}; 	\node at (10,0) {$Z$};
	\node at (12,0) {$Z$}; \node at (14,0) {$Z$};
	\node at (1,1) {$Y$}; 	\node at (5,1) {$Y$};
	\node at (9,1) {$Y$}; 	\node at (13,1) {$Y$};
	
	\draw[->,dashed] (0.7,0.7) -- node[above left, pos=0.7]
	{$\scriptstyle \phi_{0}$} (0.3,0.3);
	\draw[->,dashed] (4.7,0.7) -- node[above left, pos=0.7]
	{$\scriptstyle \phi_{1}$} (4.3,0.3);
	\draw[->,dashed] (8.7,0.7) -- node[above left, pos=0.7]
	{$\scriptstyle \phi_{n-1}$} (8.3,0.3);
	\draw[->,dashed] (12.7,0.7) -- node[above left, pos=0.7]
	{$\scriptstyle \phi_{n}$} (12.3,0.3);
	
	\draw[->,dashed] (1.3,0.7) -- node[above right, pos=0.7]
	{$\scriptstyle \psi_{1}$} (1.7,0.3);
	\draw[->,dashed] (5.3,0.7) -- node[above right, pos=0.7]
	{$\scriptstyle \psi_{2}$} (5.7,0.3);
	\draw[->,dashed] (9.3,0.7) -- node[above right, pos=0.7] {$\scriptstyle \psi_{n}$} (9.7,0.3);
	\draw[->,dashed] (13.3,0.7) -- node[above right, pos=0.7] {$\scriptstyle \psi_{n+1}$} (13.7,0.3);
	
	\draw[->,dashed] (2,1) -- node[above] {$\scriptstyle \tau_{1}$} (4,1);
	\draw[->,dashed] (5.5,1) -- node[above] {$\scriptstyle \tau_{2}$} (6.5,1);
	\draw[->,dashed] (7.5,1) -- node[above] {$\scriptstyle \tau_{n-1}$} (8.5,1);
	\draw[->,dashed] (10,1) -- node[above] {$\scriptstyle \tau_{n}$} (12,1);
	
	\draw[->,dashed] (0.5,0) -- node[below] {$\scriptstyle g_{0}$} (1.5,0);
	\draw[->,dashed] (2.5,0) -- node[below] {$\scriptstyle f_{1}$} (3.5,0);
	\draw[->,dashed] (4.5,0) -- node[below] {$\scriptstyle g_{1}$} (5.5,0);
	
	\draw[->,dashed] (8.5,0) -- node[below] {$\scriptstyle g_{n-1}$} (9.5,0);
	\draw[->,dashed] (10.5,0) -- node[below] {$\scriptstyle f_{n}$} (11.5,0);
	\draw[->,dashed] (12.5,0) -- node[below] {$\scriptstyle g_{n}$} (13.5,0);
	
	\node at (7,1) {$\cdots$};
	\node at (7,0) {$\cdots$};
\end{tikzpicture}
\end{center}
and hence we can factor $\tau$ into elementary quadratic transformations inside $\Dec(Z)$.
\end{proof}

Theorem~\ref{conicTheorem} and Theorem~\ref{cubicTheorem} follow from Proposition~\ref{mainProp}, where the three statements (A), (B), (C) appearing in the proposition are proved in each case according to:
\begin{table}[h]
\begin{center}
\renewcommand*{\arraystretch}{1.5}
\begin{tabular}{c|ccc}
 & (A) & (B) & (C) \\ \hline
Theorem~\ref{conicTheorem} & Theorem~\ref{lineTheorem} & Lemma~\ref{BforConics} & Lemma~\ref{CforConics} \\
Theorem~\ref{cubicTheorem} & Theorem~\ref{conicTheorem} & Lemma~\ref{BforCubics} & Lemma~\ref{CforCubics} 
\end{tabular}
\end{center}
\end{table}

\section{The decomposition group of a conic} \label{conicSection}

Throughout this section we let $L\subset \p^2$ denote a fixed line and $C\subset \p^2$ a conic. 

\begin{remark}
If $\phi\in\Bir(\p^2)$ is an elementary quadratic transformation belonging to $\Phi_{L,C}$ then all three base points of $\phi$ must lie outside of $L$. Conversely, given any three non-collinear points in $\p^2\setminus L$ we can always find an elementary quadratic transformation $\phi\in\Phi_{L,C}$ with these as base points.
\end{remark}

\subsection{Proof of Theorem~\ref{conicTheorem}}
We prove statements (B) \& (C) in Proposition~\ref{mainProp} in the special case that $Y=L$ a line and $Z=C$ a conic.
\subsubsection{Proof of statement (B) for conics.}
\begin{lemma} \label{BforConics}
Suppose that $\phi_1,\phi_2\in \Phi_{L,C}$. Then the composition $\phi_2\phi_1^{-1}\in\Dec(C)$ can be factored into elementary quadratic transformations inside $\Dec(C)$. 
\end{lemma}

\begin{proof}
For $i=1,2$, we let $P_i,Q_i,R_i$ be the base points of $\phi_i$, none of which lie on $L$. We may assume that these six points are in general position, i.e.\ that no points coincide and that no three points are collinear, as in Figure~\ref{conicLemmas}(i). If this is not the case, choose a third map $\phi_3\in\Phi_{L,C}$ whose base points are in general position with respect to both $\phi_1$ and $\phi_2$. Then we can write $\phi_2\phi_1^{-1} = (\phi_2\phi_3^{-1})(\phi_3\phi_1^{-1})$ and decompose each of $\phi_2\phi_3^{-1}$ and $\phi_3\phi_1^{-1}$ into elementary quadratic transformations inside $\Dec(C)$.

We let $\phi_1=:\psi_0,\psi_1,\psi_2,\psi_3:=\phi_2\in\Phi_{L,C}$ be a sequence of elementary quadratic transformations with base points:
\[ (P_1,Q_1,R_1), \; (P_1,Q_1,R_2), \; (P_1,Q_2,R_2), \; (P_2,Q_2,R_2) \]
and we write $\phi_2\phi_1^{-1} = (\psi_3\psi_2^{-1})(\psi_2\psi_1^{-1})
(\psi_1\psi_0^{-1})$.

By our assumption, $\psi_1$ and $\psi_2$ exist since no three points are collinear and we can take $\psi_1,\psi_2\in\Phi_{L,C}$ since none of these points lie on $L$. Moreover $\psi_{i+1}\psi_i^{-1}\in\Dec(C)$ is an elementary quadratic transformation for $i=0,1,2$ since $\psi_i$ and $\psi_{i+1}$ share exactly two common base points and no three base points are collinear.
\end{proof}

\begin{figure}[h]
\begin{center}
\begin{tikzpicture}[scale=1.3]
	\node at (-1.5,1) {(i)};
	\draw (0,0)--(2,2);
	\node at (0,1.5) [label={left:$P_1$}]{$\bullet$};
	\node at (1.5,0) [label={right:$Q_1$}]{$\bullet$};
	\node at (1,2) [label={above:$R_1$}]{$\bullet$};
	\draw[dashed] (0,1.5) -- (1.5,0) -- (1,2) -- cycle;
	\node at (0,1) [label={left:$P_2$}]{$\bullet$};
	\node at (2,0.5) [label={right:$Q_2$}]{$\bullet$};
	\node at (2,1.5) [label={right:$R_2$}]{$\bullet$};
	\draw[dashed] (0,1) -- (2,0.5) -- (2,1.5) -- cycle;
	
	\node at (4.5,1) {(ii)};
	\draw[thick] (6,0)--(8,2);
	\node at (7.25,1.75) [label={above:$P$}]{$\bullet$};
	\node at (8,0) [label={right:$Q$}]{$\bullet$};
	\node at (7,1) [label={left:$R$}]{$\bullet$};
	\node at (6,0.5) [label={left:$S$}]{$\bullet$};
	\draw[dashed] (7.25,1.75) -- (8,0) -- (7,1) -- cycle;
\end{tikzpicture}
\caption{Configuration of base points in proof of (i) Lemma~\ref{BforConics} and (ii) Lemma~\ref{CforConics}. }
\label{conicLemmas}
\end{center}
\end{figure}
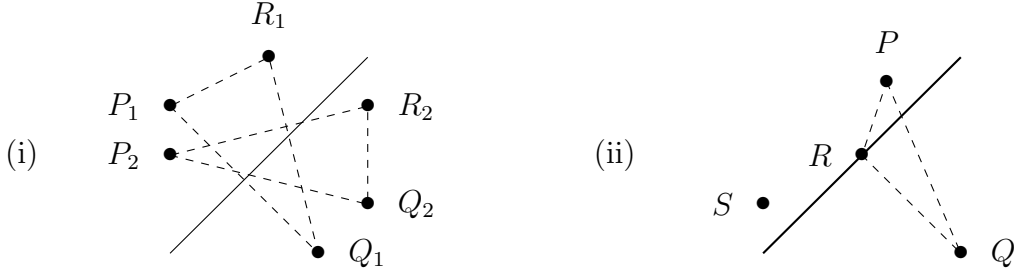

\vspace{-.2cm}
\subsubsection{Proof of statement (C) for conics.}

In fact we prove a stronger statement than statement (C) (since $\id_{\p^2}$ is a decomposition into zero elementary quadratic transformations in $\Dec(C)$).

\begin{lemma} \label{CforConics}
Let $\tau\in\Dec(L)$ be an elementary quadratic transformation. Then we can find $\phi,\psi \in\Phi_{L,C}$ such that $\phi \tau \psi^{-1}=\id_{\p^2}$.
\end{lemma}

\begin{proof}
Let $P,Q,R$ be the base points of $\tau$, where $P,Q\notin L$ and $R\in L$. Choose a point $S\notin L$ as in Figure~\ref{conicLemmas}(ii), such that no three of $P,Q,R,S$ are collinear. 

Since $P,Q,S$ are non-collinear we let $\psi\in\Phi_{L,C}$ be an elementary quadratic transformation with these base points. Then $\phi:=\psi \tau^{-1}\in\Phi_{L,C}$ is also an elementary quadratic transformation since $\psi$ and $\tau$ share two base points and no three base points are collinear. Thus $\phi \tau \psi^{-1} = \id_{\PP^2}$.
\end{proof}

\subsection{A generating set for $\Dec(C)$} \label{generatorsForDecC}

It was shown in \cite{hedzim} that, for $L\subset\PP^2$ a line, $\Dec(L)$ can be generated by $\Aut(\PP^2,L)$ and any one elementary quadratic transformation $\sigma\in\Dec(L)$. This is because $\Aut(\PP^2,L)$ is still large enough to act transitively on the set:
\[ B = \big\{ (P,Q,R)\in (\PP^2)^3 \mid P\in L \text{ and }Q,R\notin L \text{ non-collinear} \big\} \]
of all possible base points for $\sigma$. For the conic $C\subset \PP^2$, even though the analogous action of $\Aut(\PP^2,C)$ is no longer transitive, it is still true that $\Dec(C)$ can be generated by $\Aut(\PP^2,C)$ and a suitably general elementary quadratic transformation $\sigma\in\Dec(C)$. 

We fix a model $C = V\big(xz - y^2\big)\subset \PP^2$ in order to describe $\Aut(\PP^2,C)$.

\begin{lemma} \label{linearMapsInDecC}
$\Aut(\PP^2,C)$ is given by:
\[ \Aut(\PP^2,C) = \left\{ \begin{pmatrix} a^2 & 2ab & b^2 \\ ac & ad+bc & bd \\ c^2 & 2cd & d^2  \end{pmatrix} \in\PGL_3 \,\middle| \,  ad-bc \neq 0 \right\} \simeq \PGL_2\,.  \]
In particular any $\alpha\in\PGL_2=\Bir(C)$ extends uniquely to a linear map in $\Aut(\PP^2,C)$.
\end{lemma}

It follows from Lemma~\ref{linearMapsInDecC} that $\Ine(C)\cap \PGL_3=\langle\,\id_{\PP^2}\,\rangle$. Moreover the sequence
\[  1\to \Ine(C) \to \Dec(C) \to \PGL_2 \to 1  \]
is exact and $\Dec(C) = \Ine(C)\rtimes \PGL_2$ is a semidirect product, where $\PGL_2$ acts on $\Ine(C)$ by conjugation.

\begin{proposition} \label{countablyGeneratedProp}
$\Dec(C) = \langle \, \Aut(\PP^2,C), \, \sigma \, \rangle$ for any elementary quadratic transformation $\sigma$ which does not contract a tangent line to $C$.
\end{proposition}

\begin{proof}
Let $\tau\in\Dec(C)$ be an elementary quadratic transformation and consider the action of $\PGL_2\simeq\Aut(\PP^2,C)$ on the set: 
\[ B = \{ (P,Q,R)\in(\PP^2)^3 \mid P,Q\in C \text{ and } R\notin C \text{ non-collinear} \} \] 
of all possible base points for $\tau$. If $P,Q\in C$ and $R\notin C$ are the (ordered) base points of $\tau$ then, by an element of $\PGL_2$, we can send $P\mapsto(1:0:0)$, $Q\mapsto(0:0:1)$ and $R$ to a point in the conic $\Gamma_{d} = V(xz - dy^2)$ for a uniquely determined $1\neq d\in \k$. Write $B = \bigcup_{d\in \k\setminus 1} B_{d}$, a decomposition into $\PGL_2$-invariant sets according to this pencil of conics $\Gamma_d$. The sets $B_d$ with $d\neq0$ are all $\PGL_2$-orbits. For the degenerate conic $\Gamma_0$ the set $B_0$ splits into three $\PGL_2$-orbits $B_0 = B_{1,0}\cup B_{0,1}\cup B_{0,0}$ according to the cases:
\[ R\in \Gamma_{1,0} := \{(t:1:0) \mid t\neq 0\}, \quad R\in \Gamma_{0,1} := \{(0:1:t) \mid t\neq 0\}, \quad R = (0:1:0)\,. \] 
As shown in Figure~\ref{threeCases}, these three orbits correspond to the cases where one or two of the lines contracted by $\tau$ are tangent to $C$. 
\begin{figure}[h]
\begin{center}
\begin{tikzpicture}[scale=1]
	\draw (2,2) circle (1cm);
	\draw[thick] (0.8,2.7) -- (3.6,3);
	\draw[thick] (3.44,3.3) -- (2.39,0.8);
	\node at (1.35,2.75) {$\bullet$};
	\node at (2.55,1.15) {$\bullet$};
	\node at (3.3,2.96) {$\bullet$};
	\node at (0,2) {(i)};
	
	\draw (7,2) circle (1cm);
	\draw[thick] (5.8,2.7) -- (8.6,3);
	\draw[thick] (8.44,3.3) -- (7.69,0.8);
	\node at (6.35,2.75) {$\bullet$};
	\node at (8,1.75) {$\bullet$};
	\node at (8.35,2.975) {$\bullet$};
	\node at (5,2) {(ii)};
	
	\draw (12,2) circle (1cm);
	\draw[thick] (10.8,3) -- (13.6,3);
	\draw[thick] (13.44,3.3) -- (12.69,0.8);
	\node at (12,3) {$\bullet$};
	\node at (13,1.75) {$\bullet$};
	\node at (13.35,3) {$\bullet$};
	\node at (10,2) {(iii)};
\end{tikzpicture}
\caption{The base points of $\tau$ belonging to the orbit (i) $B_{d}$ with $d\neq0$, (ii) $B_{1,0}$ or $B_{0,1}$, (iii) $B_{0,0}$.}
\label{threeCases}
\end{center}
\end{figure}
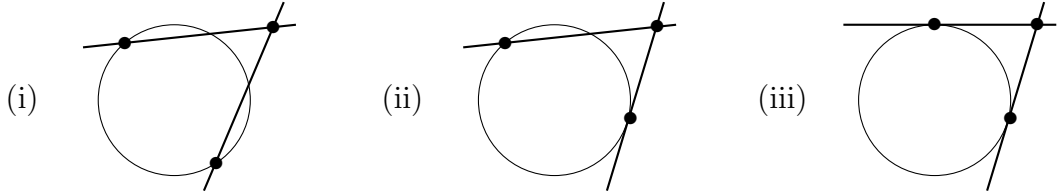

Let $\sigma_{a,b}\in\Dec(C)$ be an elementary quadratic transformation with base points $(1:0:0)$, $(0:0:1)$ and $(a:1:b)$ belonging to an orbit $B_{ab}$ with $ab\neq0$. By composing with a suitable linear map we can assume the map is actually in $\Ine(C)$, in which case $\sigma_{a,b}$ is uniquely determined and given by: 
\[ \sigma_{a,b} = \Big( \big(1-ab\big)xy + a\big(xz-y^2\big) : xz - aby^2 : \big(1-ab\big)yz + b \big(xz - y^2\big) \Big)\,.  \]

Any elementary quadratic transformation $\sigma\in\Dec(C)$ which does not contract a tangent line to $C$ has base points belonging to the same $\PGL_2$-orbit as $\sigma_{a,b}$ for some $a,b\in\k$ with $ab\neq0,1$. Therefore, to prove the proposition, it is enough to show that given any $a,b\in\k$ with $ab\neq0,1$, we can use $\sigma_{a,b}$ to generate at least one elementary quadratic transformation with base points belonging to any other $\PGL_2$-orbit.

Consider the linear map: 
\[ \lambda_{a,b} =  (x+2ay+a^2z : bx+(1+ab)y+az : b^2x+2by+z) \]
and, for $c\neq 0,1,\infty$, the diagonal map $\mu_c = (c^2x : cy : z)$. Since $ab\neq0$ we get the formula:
\[ \sigma_{a',b'} = \lambda_{a,b}^{-1} \, \mu_c^{-1} \, \sigma_{a,b} \, \mu_c \, \sigma_{a,b}^{-1} \, \lambda_{a,b} \] 
where $a'=\tfrac{1 - abc}{b(c-1)}$ and $b' = \tfrac{ab - c}{a(c-1)}$.

As $c$ varies the base points of $\sigma_{a',b'}$ are $(1:0:0)$, $(0:0:1)$ and the point $R'=\left(a(1-abc):ab(c-1):b(ab-c)\right)$ lying on the line:
\[ L_{a,b} = V\big(bx + (1+ab)y + az\big).\, \]
The point $R'$ can be any point on $L_{a,b}$, except for $(a:0:-b)$, corresponding to $c=1$, and $L_{a,b}\cap C = \{ (-\tfrac{1}{b}:1:-b), (-a:1:-\tfrac{1}{a}) \}$, corresponding to $c=0,\infty$. Outside of these points $L_{a,b}$ intersects every conic $\Gamma_{d}$ at least once. 

For all $d\neq0$ this construction gives an elementary quadratic transformation with base points in $B_d$. 

If $d=0$ and $ab\neq -1$ then $L_{a,b}$ meets $\Gamma_{1,0}$ and $\Gamma_{0,1}$ giving elementary quadratic transformations with base points in $B_{1,0}$ and $B_{0,1}$. If $ab = -1$ then $L_{a,b}\cap\Gamma_0=(0:1:0)$ giving an elementary quadratic transformation with base points in $B_{0,0}$.

It remains to produce an elementary quadratic transformation with base points in $B_{0,0}$ if $ab\neq -1$ and in $B_{1,0}$ and $B_{0,1}$ if $ab=-1$. We can use the construction once to produce $\sigma_{a',b'}$ with $a'b'=-1$ if $ab\neq -1$ (or with $a'b'\neq-1$ if $ab=-1$) and then proceed as above.
\end{proof}

\begin{remark} \label{uncountableIneC}
If the ground field $\k$ is uncountable then the corresponding statement for $\Ine(C)$ is not true, i.e.\ $\Ine(C)$ cannot be generated by linear maps and any countable collection of elementary quadratic maps. Although $\Ine(C)\cap\PGL_3$ is trivial, $\Ine(C)$ contains a lot of elementary quadratic transformations. Indeed the maps $$\{\sigma_{a,b}\in\Ine(C)\mid a,b\in \k, \, ab\neq1\}$$ appearing in the proof of Proposition~\ref{countablyGeneratedProp} give an uncountable family. 
\end{remark}

\section{The decomposition group of a rational cubic} \label{cubicSection}

Throughout this section we let $C\subset \PP^2$ denote a fixed conic and $X\subset \PP^2$ a rational cubic. We will distinguish between the nodal and cuspidal cases when necessary. As explained in Remark~\ref{notChar2}, we will also assume that the characteristic of $\k$ is not 2.

\begin{remark}
Any map $\phi\in\Phi_{C,X}$ must have exactly one base point $P\in C$ and two base points $Q,R\notin C$. In this case $X$ is a cuspidal cubic if the line $\overline{QR}$ is tangent to $C$ and a nodal cubic otherwise, as shown in Figure \ref{basePointConfigurations}. Moreover, given any three non-collinear points in such a position we can always find a map $\phi\in\Phi_{C,X}$ with these three points as base points.
\end{remark}
\begin{figure}[h]
\begin{center}
\begin{tikzpicture}[scale=1]
	\draw[thick] (0,0) circle (1cm);
	\node at (0,-1) [label = {below:$P$}]{$\bullet$};
	\node at (1.2,0.7) [label = {above:$R$}]{$\bullet$};
	\node at (-1.2,0.7) [label = {above:$Q$}]{$\bullet$};
	\draw[thick, dashed] (0,-1) -- (-1.2,0.7) -- (1.2,0.7) -- cycle;
	
	\draw[thick] (6,0) circle (1cm);
	\node at (6,-1) [label = {below:$P$}]{$\bullet$};
	\node at (7,1) [label = {above:$R$}]{$\bullet$};
	\node at (5,1) [label = {above:$Q$}]{$\bullet$};
	\draw[thick, dashed] (6,-1) -- (7,1) -- (5,1) -- cycle;
	
	\node at (-2,0) {(i)};
	\node at (4,0) {(ii)};
\end{tikzpicture}
\caption{Base point configurations for $\phi \in \Phi_{C,X}$ when $X$ is (i) a nodal cubic and (ii) a cuspidal cubic.}
\label{basePointConfigurations}
\end{center}
\end{figure}
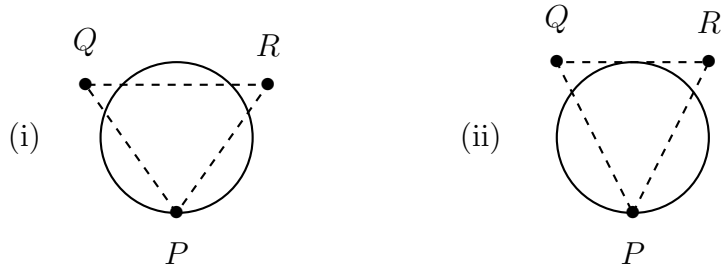

\subsection{Proof of Theorem \ref{cubicTheorem}}

We now prove statements (B) \& (C) in Proposition \ref{mainProp} for $Y=C$ a conic and $Z=X$ a rational cubic. 

\subsubsection{Proof of statement (B) for cubics}

\begin{lemma} \label{BforCubics}
Suppose that $\phi_1,\phi_2\in \Phi_{C,X}$. Then the composition $\phi_2\phi_1^{-1}\in\Dec(X)$ can be factored into elementary quadratic transformations inside $\Dec(X)$. 
\end{lemma}

\begin{proof}
For $i=1,2$, we let $P_i,Q_i,R_i$ be the base points of $\phi_i$, where $P_i\in C$ and $Q_i,R_i\notin C$. As in the proof of Lemma \ref{BforConics}, we may intertwine with a third map $\phi_3\in\Phi_{C,X}$ to assume that no base points coincide, no three are collinear and no two lie on a tangent line to $C$ (unless $X$ is a cuspidal cubic, in which case we can assume that only $Q_1,R_1$ and $Q_2,R_2$ lie on a tangent line to $C$). 

{\it The nodal case:} If $X$ is a nodal cubic we let $\phi_1=:\psi_0,\psi_1,\psi_2,\psi_3:=\phi_2\in\Phi_{C,X}$ be a sequence of elementary quadratic transformations with base points:
\[ (P_1,Q_1,R_1), \; (P_1,Q_1,R_2), \; (P_1,Q_2,R_2), \; (P_2,Q_2,R_2) \]
and we write $\phi_2\phi_1^{-1} = (\psi_3\psi_2^{-1})(\psi_2\psi_1^{-1})(\psi_1\psi_0^{-1})$.

By our assumption $\psi_1$ and $\psi_2$ exist since each of these triples is non-collinear and $\psi_1,\psi_2\in\Phi_{C,X}$ since they both have precisely one base point on $C$ and do not contract any tangent line to $C$. Lastly each composition $\psi_{i+1}\psi_i^{-1}\in\Dec(X)$ is an elementary quadratic transformation since $\psi_i$ and $\psi_{i+1}$ share exactly two common base points and no three base points are collinear.

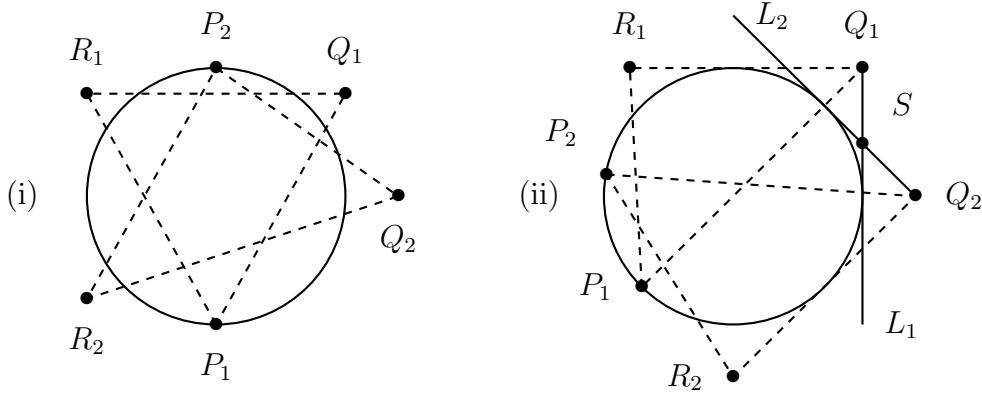
\begin{figure}[h]
\begin{center}
\begin{tikzpicture}[scale=1.7]
	\draw[thick] (0,0) circle (1cm);
	\node at (0,-1) [label={below:$P_1$}]{$\bullet$};
	\node at (1,0.8) [label={above:$Q_1$}]{$\bullet$};
	\node at (-1,0.8) [label={above:$R_1$}]{$\bullet$};
	\node at (0,1) [label={above:$P_2$}]{$\bullet$};
	\node at (1.41,0) [label={below:$Q_2$}]{$\bullet$};
	\node at (-1,-0.8) [label={below:$R_2$}]{$\bullet$};
	\draw[thick, dashed] (0,-1) -- (-1,0.8) -- (1,0.8) -- cycle;
	\draw[thick, dashed] (0,1) -- (1.41,0) -- (-1,-0.8) -- cycle;
	\node at (-1.5,0){(i)};
	
	\draw[thick] (4,0) circle (1cm);
	\node at (4-0.707,-0.707) [label={left:$P_1$}]{$\bullet$};
	\node at (5,1) [label={above:$Q_1$}]{$\bullet$};
	\node at (3.2,1) [label={above:$R_1$}]{$\bullet$};
	\node at (3.02,0.17) [label={135:$P_2$}]{$\bullet$};
	\node at (5.41,0) [label={right:$Q_2$}]{$\bullet$};
	\node at (4,-1.41) [label={left:$R_2$}]{$\bullet$};
	\draw[thick, dashed] (4-0.707,-0.707) -- (3.2,1) -- (5,1) -- cycle;
	\draw[thick] (5,1) -- (5,-1);
	\draw[thick, dashed] (3.02,0.17) -- (5.41,0) -- (4,-1.41) -- cycle;
	\draw[thick] (5.41,0) -- (4,1.41);
	\node at (5,0.41) [label={45:$S$}]{$\bullet$};
	\node at (5,-1) [label={right:$L_1$}]{};
	\node at (4,1.41) [label={right:$L_2$}]{};
	\node at (2.5,0){(ii)};
\end{tikzpicture}
\caption{Configuration of base points in (i) the nodal case and (ii) the cuspidal case.}
\label{default}
\end{center}
\end{figure} 

{\it The cuspidal case:} If $X$ is a cuspidal cubic then we must be a little bit more careful to ensure that each of our intermediate maps $\psi_i$ contracts a tangent line to $C$.

For $i=1,2$ let $L_i$ be the tangent line to $C$ passing through $Q_i$ which does not contain $R_i$. By our assumption on the position of the base points, the point $S = L_1\cap L_2$ is well-defined, $S\notin C$ and $S$ is not equal to any $P_i,Q_i,R_i$. Moreover, no three of the seven points $P_1,P_2,Q_1,Q_2,R_1,R_2,S$ are collinear.

Now we let $\phi_1=:\psi_0,\psi_1,\psi_2,\psi_3,\psi_4:=\phi_2\in\Phi_{C,X}$ be a sequence of elementary quadratic transformations with base points:
\[ (P_1,Q_1,R_1), \; (P_1,Q_1,S), \; (P_1,Q_2,S), \; (P_2,Q_2,S), \; (P_2,Q_2,R_2) \]
and we write $\phi_2\phi_1^{-1} = (\psi_4\psi_3^{-1})(\psi_3\psi_2^{-1})(\psi_2\psi_1^{-1})(\psi_1\psi_0^{-1})$.

As before, $\psi_1,\psi_2,\psi_3$ exist since each triple of base points is non-collinear and $\psi_1,\psi_2,\psi_3\in\Phi_{C,X}$ since they all have precisely one base point on $C$ and contract a tangent line to $C$. Lastly each composition $\psi_{i+1}\psi_i^{-1}\in\Dec(X)$ is an elementary quadratic transformation since $\psi_i$, $\psi_{i+1}$ share exactly two common base points and no three base points are collinear.
\end{proof}

\subsubsection{Proof of statement (C) for cubics}

\begin{lemma} \label{CforCubics}
Let $\tau\in\Dec(C)$ be an elementary quadratic transformation. Then we can find $\phi,\psi \in\Phi_{C,X}$ such that $\phi \tau \psi^{-1}\in\Dec(X)$ can be factored into elementary quadratic transformations inside $\Dec(X)$.
\end{lemma}

\begin{proof}
We first assume that $\tau$ is an elementary quadratic transformations which does not contract a tangent line to $C$ (i.e.\ $\tau$ has a configuration of base points as in Figure~\ref{threeCases}(i)). Let $P,Q\in C$ and $R\notin C$ be the base points of $\tau$ and let $L$ be a tangent line to $C$ passing through $R$. By assumption $L\neq\overline{PR},\overline{QR}$.

Choose a point $S\notin C$ as in Figure \ref{choiceOfS}, such that no three of $P,Q,R,S$ are collinear. If $X$ is a nodal cubic then we choose $S$ to avoid the tangent lines to $C$ passing through $P$, $Q$ or $R$. If $X$ is a cuspidal cubic then we choose $S$ to lie on $L$ but avoid the tangent lines to $C$ through $P$ or $Q$.
\begin{figure}[h]
\begin{center}
\begin{tikzpicture}[scale=1.3]
	\draw (2,2) circle (1cm);
	\draw[thick] (0.8,2.7) -- (3.6,3);
	\draw[thick] (3.44,3.3) -- (2.39,0.8);
	\node at (1.35,2.75) [label={above:$P$}]{$\bullet$};
	\node at (2.55,1.15) [label={right:$Q$}]{$\bullet$};
	\node at (3.3,2.96) [label={45:$R$}]{$\bullet$};
	\node at (0.7,1.5) [label={left:$S$}]{$\bullet$};
	\node at (0,2) {(i)};
	
	\draw (8,2) circle (1cm);
	\draw[thick] (6.8,2.7) -- (9.6,3);
	\draw[thick] (9.44,3.3) -- (8.39,0.8);
	\draw[thick, dashed] (9.3,3) -- (6,3);
	\node at (7.35,2.75) [label={below:$P$}]{$\bullet$};
	\node at (8.55,1.15) [label={right:$Q$}]{$\bullet$};
	\node at (9.3,2.96) [label={45:$R$}]{$\bullet$};
	\node at (6.2,3) [label={above:$S$}]{$\bullet$};
	\node at (6,2) {(ii)};
\end{tikzpicture}
\caption{Location of the point $S$ when $X$ is (i) a nodal cubic and (ii) a cuspidal cubic. }
\label{choiceOfS}
\end{center}
\end{figure}
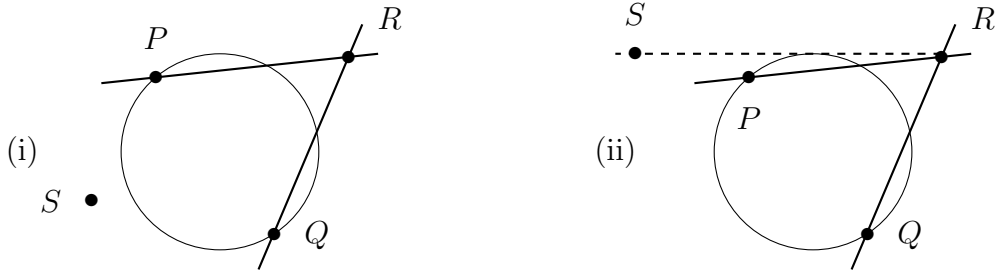

Since $P,R,S$ are non-collinear there is an elementary quadratic transformation $\psi\in\Phi_{C,X}$ with these base points. We let $\phi:=\psi \tau^{-1}\in\Phi_{C,X}$ which is also an elementary quadratic transformation since $\psi$ and $\tau$ share two base points and no three of the base points are collinear. Thus $\phi \tau \psi^{-1} = \id_{\PP^2}\in\Dec(X)$ which is a decomposition into zero elementary quadratic transformations inside $\Dec(X)$.

If $\tau$ is an arbitrary elementary quadratic transformation, then by Proposition~\ref{countablyGeneratedProp} we can write $\tau=\tau_n\cdots\tau_1$ where $\tau_i\in\Dec(C)$ are elementary quadratic transformations which do not contract a tangent line to $C$. We can find $\phi_i,\psi_i\in\Phi_{C,X}$, for $i=1,\ldots,n$, such that $\phi_i \tau_i \psi_i^{-1}\in\Dec(X)$ can be factored into elementary quadratic transformations inside $\Dec(X)$ and by Lemma~\ref{BforCubics} we can factor $\psi_{i+1}\phi_i^{-1}\in\Dec(X)$ into elementary quadratic transformations inside $\Dec(X)$ for $i=1,\ldots,n-1$. Therefore, taking $\phi:=\phi_n$ and $\psi:=\psi_1$, we can factor
\[ \phi\tau\psi^{-1} = (\phi_n\tau_n\psi_n^{-1})(\psi_n\phi_{n-1}^{-1})(\phi_{n-1}\tau_{n-1}\psi_{n-1}^{-1})\cdots(\psi_2\phi_1^{-1})(\phi_1\tau_1\psi_1^{-1}) \]
 into elementary quadratic transformations inside $\Dec(X)$.
\end{proof}

\subsection{An example}

Let $X$ be a nodal (resp.\ cuspidal) cubic, let $\tau\in\Dec(X)$ and suppose that we conjugate $\tau$ to get $\tau'\in\Dec(C)$, for a conic $C$, as in the proof of Proposition~\ref{mainProp}. If $\tau'$ can be decomposed into $n$ elementary quadratic transformations which do not contract any tangent line to $C$ then na\"ively applying the proof of Theorem \ref{cubicTheorem} gives a decomposition of $\tau$ into at most $6(n+1)$ (resp.\ $8(n+1)$) elementary quadratic transformations inside $\Dec(X)$. 

Even in relatively simple cases this gives a very long decomposition which is far from optimal. For example let $X$ be the cuspidal cubic $X = V(x^3 - y^2z)\subset \PP^2$ and consider the de Jonqui\`eres involution $\tau = (xy^2:y^3:2x^3-y^2z)\in\Ine(X)$. This map has one proper base point at the cusp point $P\in X$ and all other base points infinitely near to $P$. If $C$ is the conic $C=V(xz-y^2)$ then $\phi=(x(y+z):x(x+y):z(y+z))\in\Phi_{C,X}$ and conjugating $\tau$ with $\phi$ gives $\tau' = \phi^{-1}\tau\phi\in\Dec(C)$, a map of degree 3 with two proper base points, which decomposes into four elementary quadratic transformations in $\Dec(C)$ not contracting any tangent line to $C$. Therefore we can decompose $\tau$ into at worst 40 elementary quadratic transformations inside $\Dec(X)$, although we expect a minimal decomposition to be much shorter.

\subsection{Generating sets for $\Dec(X)$} \label{generatingCubic}

Let $X$ be the nodal cubic given by the model $X=V(x^3+y^3-xyz)\subset\PP^2$. We see that $\Aut(\PP^2,X)$ is the finite group given by:
\[ \Aut(\PP^2,X) = \left\langle \begin{pmatrix} \omega & 0 & 0 \\ 0 & \omega^2 & 0 \\ 0 & 0 & 1  \end{pmatrix},  \begin{pmatrix} 0 & 1 & 0 \\ 1 & 0 & 0 \\ 0 & 0 & 1  \end{pmatrix} \right\rangle \simeq S_3 \]
where $\omega\in \k$ is a primitive cube root of unity. If $\k$ is an uncountable field then $\Dec(X)$ is an uncountable group and therefore cannot be generated by $\Aut(\PP^2,X)$ and any finite (or countable) collection of elementary quadratic transformations.

Now suppose $X$ is the cuspidal cubic given by the model $X=V(x^3-y^2z)\subset\PP^2$. In this case $\Aut(\PP^2,X)$ is infinite:
\[ \Aut(\PP^2,X) = \left\langle \begin{pmatrix} a & 0 & 0 \\ 0 & 1 & 0 \\ 0 & 0 & a^3  \end{pmatrix} \middle| \, a\in \k^\times \right\rangle \simeq \mathbb{G}_m\,. \]
We do not know whether or not $\Dec(X)$ can be generated by $\Aut(\PP^2,X)$ and any countable collection of elementary quadratic transformations.

\section{Rational curves of higher degree}

We provide a family of plane rational curves $X_d\subset \PP^2$, birationally equivalent to a line and of degree $d\geq4$, to show that we cannot expect Theorems~\ref{lineTheorem},~\ref{conicTheorem}~\&~\ref{cubicTheorem} to be true for curves of higher degree.

Let $X_d$ denote the rational curve given by $X_d=V(x^d-y^{d-1}z)\subset\p^2$ which has a unique singular point $P=(0:0:1)$, a cusp of multiplicity $d-1$, and a unique inflection point $Q=(0:1:0)$. Let $L_Q=(z=0)$ be the tangent line intersecting $X_d$ at $Q$ with multiplicity $d$ and let $L_P=(y=0)$ be the tangent line to the cusp $P$. Any de Jonqui\`eres transformation of degree $d$ with major base point at $P$ and all other base points on $X_d\setminus P$ sends $X_d$ onto a line.

A map in $\Aut(\p^2,X_d)$ has to fix $P$ and $Q$ and preserve $L_P$ and $L_Q$. It is straightforward to check that:
\[ \Aut(\p^2,X_d)=\left\{ (a x : y : a^d z)\,\middle| \, a\in\k^\times \right\} \simeq\mathbb{G}_m\,. \]

\begin{lemma}\label{Lem:Xd}
The standard involution $\sigma = (yz:zx:xy) \in\Bir(\p^2)$ is the only elementary quadratic map that preserves $X_d$, up to composition with an element of $\Aut(\p^2,X_d)$.
\end{lemma}

\begin{proof} 
It is easy to check that $\sigma\in\Dec(X_d)$. Any other elementary quadratic transformation $\tau\in\Dec(X_d)$ must have one base point at $P\in X_d$, one base point in the smooth locus of $X_d$ and one base point not contained in $X_d$. In particular $\tau^{-1}$ also has a base point at $P$. Since the line $\tau^{-1}(P)$ is tangent to a point of $X_d$ with multiplicity $\geq d-1$, we must have $\tau^{-1}(P) = L_Q$. As the line $L_Q$ is contracted, both $\tau$ and $\tau^{-1}$ must have two base points on $L_Q$, one of which is $L_Q\cap X_d=Q$. Now the line $\tau^{-1}(Q)$ is tangent to the cusp $P$ so we must have $\tau^{-1}(Q) = L_P$, as in Figure~\ref{Fig:Xd}. 

Since the lines $L_P$ and $L_Q$ are contracted, the base points of $\tau$ are $P=(0:0:1)$, $Q=(0:1:0)$ and $L_P\cap L_Q=(1:0:0)$. Hence, up to an element of $\Aut(\p^2,X_d)$, we must have $\tau=\sigma$.
\end{proof}

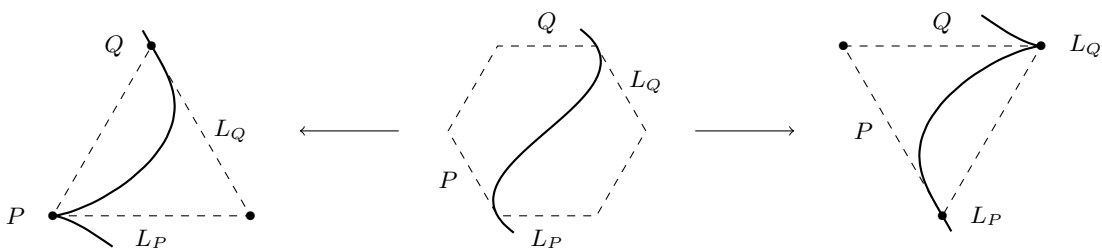
\begin{figure}[h]
\begin{center}
\begin{tikzpicture}[scale=1.3]
	\node at (0,{1-sqrt(3)/2}) [label={left:\scriptsize $P$}] {\tiny $\bullet$};
	\node at (2,{1-sqrt(3)/2}) {\tiny $\bullet$};
	\node at (1,{1+sqrt(3)/2}) [label={left:\scriptsize $Q$}] {\tiny $\bullet$};
	\node at (1.8,1) {\scriptsize $L_Q$};
	\node at (1,-0.1) {\scriptsize $L_P$};
	\draw[dashed] (0,{1-sqrt(3)/2}) -- (2,{1-sqrt(3)/2}) -- (1,{1+sqrt(3)/2}) -- cycle;
	\draw[scale=1,domain=-0.3:1.2,smooth,variable=\x,thick] plot ({8*\x*\x/(3*\x^3+7*\x^2-3*\x+1)},{1 - sqrt(3)/2 + 8*sqrt(3)*\x*\x*\x/(3*\x^3+7*\x^2-3*\x+1)});
	
	\draw[<-] (2.5,1)--(3.5,1);
	
	\node at (5,2.1) {\scriptsize $Q$};
	\node at (6,1.5) {\scriptsize $L_Q$};
	\node at (5,-0.1) {\scriptsize $L_P$};
	\node at (4,0.5) {\scriptsize $P$};
	\draw[dashed] ({5+cos(0)},{1+sin(0)}) -- ({5+cos(60)},{1+sin(60)}) -- ({5+cos(120)},{1+sin(120)}) -- ({5+cos(180)},{1+sin(180)}) -- ({5+cos(240)},{1+sin(240)}) -- ({5+cos(300)},{1+sin(300)}) -- cycle;
	\draw[scale=1,domain=-1.2:1.2,smooth,variable=\x,thick] plot ({5 + \x - \x*\x*\x/2},{1 + sqrt(3)*\x/2});
	
	\draw[->] (6.5,1)--(7.5,1);
	
	\node at (8,{1+sqrt(3)/2}) {\tiny $\bullet$};
	\node at (10,{1+sqrt(3)/2})[label={right:\scriptsize $L_Q$}] {\tiny $\bullet$};
	\node at (9,{1-sqrt(3)/2})[label={right:\scriptsize $L_P$}] {\tiny $\bullet$};
	\node at (8.2,1) {\scriptsize $P$};
	\node at (9,2.1) {\scriptsize $Q$};
	\draw[dashed]  (8,{1+sqrt(3)/2}) -- (10,{1+sqrt(3)/2}) -- (9,{1-sqrt(3)/2}) -- cycle;
	\draw[scale=1,domain=-0.3:1.2,smooth,variable=\x,thick] plot ({10 - 8*\x*\x/(3*\x^3+7*\x^2-3*\x+1)},{1 + sqrt(3)/2 - 8*sqrt(3)*\x*\x*\x/(3*\x^3+7*\x^2-3*\x+1)});
\end{tikzpicture}
\caption{Resolution of the standard involution $\sigma\in\Dec(X_d)$.}
\label{Fig:Xd}
\end{center}
\end{figure}

\begin{proposition}\label{Prop:Xd}
If $d\geq4$, the group $\Dec(X_d)$ cannot be generated by linear maps and elementary quadratic transformations.
\end{proposition}

\begin{proof}
By Lemma~\ref{Lem:Xd}, the subgroup of $\Dec(X_d)$ generated by linear maps and elementary quadratic transformations is given by $\langle \, \Aut(\p^2,X_d), \, \sigma \, \rangle$. Since $\sigma^2=\id_{\PP^2}$ and $\sigma\lambda = \lambda^{-1}\sigma$ for any $\lambda\in\Aut(\p^2,X_d)$, all elements of this subgroup are of the form $\lambda$ or $\lambda\sigma$ and are either linear or quadratic. But there are many elements in $\Dec(X_d)$ of degree $>2$; for example the de Jonqui\`eres transformation $\tau_a = (xy^{d-1}:y^d:(1-a)x^d+ay^{d-1}z)$ for $a\in\k^\times$.
\end{proof}

\begin{remark}
The family of maps $\{\tau_a \mid a\in\k^\times\}$, appearing at the end of the proof of Proposition~\ref{Prop:Xd}, form a subgroup of $\Ine(X_d)$ isomorphic to $\G_m$ since $\tau_b \tau_a = \tau_{ab}$ for all $a,b\in\k^\times$.
\end{remark}

\end{document}